\documentclass[12pt]{article}

\usepackage{graphicx}%
\usepackage{multirow}%
\usepackage{amsmath,amssymb,amsfonts}%
\usepackage{amsthm}%
\usepackage{mathrsfs}%
\usepackage[title]{appendix}%
\usepackage{xcolor}%
\usepackage{textcomp}%
\usepackage{manyfoot}%
\usepackage{booktabs}%
\usepackage{algorithm}%
\usepackage{algorithmicx}%
\usepackage{algpseudocode}%
\usepackage{listings}%
\usepackage[top=2cm, bottom=2cm, right=2cm, left=2cm]{geometry}

\newtheorem{definition}{Definition}
\newtheorem{proposition}{Proposition}
\newtheorem{sled}{Corollary}

\newtheorem{lemma}{Lemma}
\newtheorem*{lemma*}{Lemma}
\newtheorem{theor}{Theorem}
\newtheorem*{theor*}{Theorem}
\newtheorem{example}{Example}
\newtheorem{remark}{Remark}
\newcommand{\sign}{\operatorname{sign}}

\newcommand{\SL}{\operatorname{SL}}

\newcommand{\wt}{\operatorname{wt}}
\newcommand{\ch}{\operatorname{ch}}
\theoremstyle{definition}
\date{}
\begin{document}
\title {Branching rule for $\SL_{n+m}(\mathbb{C})\supset\SL_n(\mathbb{C})\times\SL_m(\mathbb{C})$}
\author{Andrei Gornitskii}

\maketitle

\abstract{
We study the branching problem for the pair $\SL_{n+m}(\mathbb{C})\supset\SL_n(\mathbb{C})\times\SL_m(\mathbb{C})$. We describe the corresponding branching rule in terms of a semigroup $\Sigma_{n,m}\subset\Lambda^{+}\times\mathbb{Z}^N$, where $\Lambda^{+}$ is the semigroup of dominant weights of $\SL_{n+m}$, and $N$ is the dimension of maximal unipotent subgroup in $\SL_{n+m}$. Let $V(\lambda)$ be the irreducible representation of $\SL_{n+m}$ with the highest weight $\lambda$. For every dominant weight $\lambda\in\Lambda^{+}$ the set of all $\sigma\in\Sigma_{n,m}$ with dominant weight $\lambda$ parametrizes the irreducible representations in the restriction $V(\lambda)|_{\SL_n\times\SL_m}$ of $V(\lambda)$ to $\SL_{n}\times\SL_{m}$. We describe the semigroup $\Sigma_{n,m}$ as the semigroup of integral points in some polyhedral cone and we find the inequalities defining this cone. }

\section{Introduction} 

The \emph{branching problem} concerns the decomposition of an irreducible representation of a reductive group $G$ upon restriction to a subgroup $H$. A description of \emph{branching rules } (i.e. a solution of the branching problem) is essential for understanding how representations behave under symmetry reduction. Of special interest is the case where $H$ is a maximal subgroup of $G$. It is well known that maximal regular (i.e. stable under the action of the maximal torus consisting of diagonal matrices) subgroups of $G=\SL_{n}(\mathbb{C})$ are $H=\SL_{n-k}(\mathbb{C})\times\SL_{k}(\mathbb{C})$, where the subgroup is block-diagonal. Despite being a classical subject, a solution is unknown in general. 

 If a highest weight $\lambda$ of $G$ is fixed, then there are a number of approaches to solve branching problem for this particular highest weight. For example, one can use Weyl's character formula and find the decomposition of the restriction of the character $\ch V(\lambda)$ on subgroup $H$ as the linear combination of characters of $H$. This approach is unsatisfactory, as it does not provide a connection for different $\lambda$.

It turns out that the irreducible representations of $H$ in the restriction of $V(\lambda)$ (for all dominant $\lambda$) can be parametrized by a certain semigroup, which we call a \emph{branching semigroup}. This allow us to relate the branching rules for different $\lambda$. 

Now we state our main result. We denote by $\varepsilon_i, i=1,\ldots,n+m,$ the weights of the standard representation of $\SL_{n+m}$ in $\mathbb{C}^{n+m}$. Here $\varepsilon_i$ is the weight of the vector $e_i$ with respect to the maximal torus $\SL_{n+m}$ consisting of diagonal matrices. Then the root system is $\Delta=\{\varepsilon_i-\varepsilon_j|i\neq j, 1\leq i,j\leq n+m\}$. Let $V(\lambda)$ be the irreducible representation of $SL_{n+m}$ with the highest weight $\lambda$. We have the decomposition $\lambda=\sum_{i=1}^{n+m-1}k_i\omega_i$, where $\omega_i$ are the fundamental weights of $\SL_{n+m}$ (here $V(\omega_i)=\bigwedge^i (\mathbb{C}^{n+m})$).

We consider the branching problem for the pair $G=\SL_{n+m}$ and $H=\SL_{n}\times\SL_{m}$ and we denote the branching semigroup (see \ref{branching semigroup}) by $\Sigma_{m,n}$. Let $p_{i,j}$ be entries of strictly lower triangular $(n+m)\times (n+m)$ matrix.   We set $v_{i,j}=p_{i,j}+p_{i-1,j}+\ldots+p_{j+1,j}-p_{i-1,j+1}-p_{i-2,j+1}-\ldots-p_{j+2,j+1}$ and $h_{i,j}=p_{i,j}+p_{i,j+1}+\ldots+p_{i,i-1}-p_{i-1,j+1}-\ldots-p_{i-1,i-2}$. The main theorem of the paper is the following
\begin{theor*}
The branching semigroup $\Sigma_{n,m}$ is the semigroup of integral points in the polyhedral cone given by the following inequalities and equalities:
\begin{enumerate}
\item[0.] $p_{i,j}\geq 0$,
\item[1.] $p_{i,j}=0,\quad i<n+1\quad \textrm{or}\quad j>n$,
\item[2.] $ v_{i,j}\leq k_j,\quad 1\leq j\leq n-1$,
\item[3.] $ h_{i,j}\leq k_{i-1},\quad n+2\leq i\leq n+m$,
\item[4.] $ p_{n+1,n}\leq k_n.$
\end{enumerate}
\end{theor*}
If $\sigma=(\lambda;\textbf{p})$ is an integral solution of the above inequalities and equalities, where $\textbf{p}$ is the set of $p_{i,j}$, then \emph{the weight} of $\sigma$ is $\wt(\sigma)=\lambda+\sum_{(i,j)}p_{i,j}(\varepsilon_i-\varepsilon_j)$ and \emph{the highest weight} of $\sigma$ is $\lambda$. From the definition of a branching semigroup and its basic properties, the branching rule follows: 
$$
V(\lambda)|_{\SL_{n}\times\SL_m}=\bigoplus_{\sigma\in\Sigma(\lambda)} V(\wt(\sigma)|_{\SL_n\times\SL_m}),
$$
where $\wt(\sigma)|_{\SL_n\times\SL_m}$ denotes the restriction of the weight $\wt(\sigma)$ to ${\SL_{n}\times\SL_m}$ and $\Sigma(\lambda)$ denotes the set of all (integral) solutions with the highest weight $\lambda$.

 The approach we use is general and was described in our paper \cite{[G3]}. It allows one to construct, from an essential semigroup $\Sigma$ parametrizing bases in irreducible representations of a simply connected semisimple group $G$, the branching semigroup that describes the branching rule from $G$ to $H$. For the reader's convenience, we have included sections with precise definitions (see \ref{essential semigroup}, \ref{U-invariant functions}, \ref{branching problem}, and \ref{branching semigroup}). For our purposes, we use the essential semigroup for $\SL_{n+m}$, described in the work \cite{[FFL1]} (see \ref{dyck paths}). We then describe the branching semigroup, proving that it is precisely the semigroup of integral points in the polyhedral cone from Theorem above. The proof of the main theorem relies on a technical lemma, and we have dedicated a separate section to its proof (see \ref{proof of lemma}).

 \section{Bases in the irreducible representations of $\mathfrak{g}$}
 
 All definitions and results of \ref{essential semigroup} and \ref{U-invariant functions} are due to Vinberg.
 \subsection{Essential semigroups}\label{essential semigroup}
 We recall the notion of essential semigroup that parametrizes certain bases in irreducible finite-dimensional representations of a simple complex Lie algebra.

 Let $\mathfrak{g}$ be a simple Lie algebra with the triangular decomposition
$\mathfrak{g}=\mathfrak{u}^{-}\oplus\mathfrak{t}\oplus\mathfrak{u}$, where
$\mathfrak{u}^{-}$ and $\mathfrak{u}$ are mutually opposite maximal unipotent subalgebras, and $\mathfrak{t}=\mathfrak{t}_{\mathfrak{g}}$ is a Cartan subalgebra.

One has: $\mathfrak{u}= \langle e_{\alpha}$ $\mid$ $\alpha$ $\in$ $\Delta_{+}\rangle$,
$\mathfrak{u^{-}}=\langle e_{-\alpha}$ $\mid$ $\alpha$ $\in$ $\Delta_{+}\rangle$, where
$\Delta_{+}=\Delta_+{(\mathfrak{g})}$ is the system of positive roots, $e_{\pm\alpha}$
 are the root vectors, and the symbol $\langle\ldots\rangle$ stands for the linear span.

 We denote the finite-dimensional irreducible $\mathfrak{g}$-module with highest weight $\lambda$ by $V(\lambda)$ or $V_{\mathfrak{g}}(\lambda)$ and a highest weight vector in this module by $v_{\lambda}$. We fix an ordering of the positive roots : $\Delta_{+}=\{\alpha_{1},\dots,\alpha_{N}\}$.

\begin{definition}
 A signature is an $(N+1)$-tuple
 $\sigma=(\lambda;p_{{1}},\dots,p_{{N}})$,
  where $\lambda$ is a dominant weight, and $p_{i}\in\mathbb{Z}_{+}$.
\end{definition}

Set
$$
{v}(\sigma)=e_{-\alpha_{1}}^{p_{{1}}}\cdot \ldots \cdot
 e_{-\alpha_{N}}^{p_{{N}}}\cdot{v}_{\lambda}\in V(\lambda).
$$
$\lambda$ is called the \emph{highest weight} of $\sigma$, the eigenweight $\wt(\sigma)=\lambda-\sum p_i\alpha_i$ of the vector $v(\sigma)$ is called the \emph{weight} of $\sigma$, and the numbers $(p_1,\ldots,p_N)$ are called \emph{exponents} of $\sigma$.

 Fix any monomial order $<$ on $\mathbb{Z}^N$. We use this order to compare signatures with the same highest weight $\lambda$, i.e. if $\sigma=(\lambda;p_1,\ldots,p_N)$ and $\tau=(\lambda;q_1,\ldots,q_N)$, then $$\sigma<\tau \iff (p_1,\ldots,p_N)<(q_1,\ldots,q_N).$$

 \begin{definition} A signature $\sigma$ is essential if
 $v(\sigma)\notin\langle v(\tau)\mid\tau<\sigma\rangle$.

\end{definition}

For a dominant weight $\lambda$ the set $\{{v}(\sigma)\mid\sigma \mbox{ is essential of highest weight $\lambda$}\}$ is a basis of $V(\lambda)$. Moreover, the set of essential signatures (for all $\lambda$) is a subsemigroup of $\Lambda^+\times\mathbb{Z}_{+}^N$, where $\Lambda^+$ is the semigroup of dominant weights. The proof will be given below. We denote the semigroup of essential signatures, or \emph{essential semigroup}, by $\Sigma$. Let $\Sigma(\lambda)$ be the set of all essential signatures of the highest weight $\lambda$.

\subsection{U-invariant functions}\label{U-invariant functions}

Let $G$ be a simply connected simple complex algebraic group such that $\mathop{\mathrm{Lie}}G=\mathfrak{g}$.  Let $T$ be the maximal torus in $G$ such that $\mathop{\mathrm{Lie}}T=\mathfrak{t}$ and $U$ be the maximal unipotent subgroup of $G$ such that
 $\mathop{\mathrm{Lie}}U=\mathfrak{u}$.
 
Now we show that the essential signatures can be interpreted as least terms of functions on the homogeneous space $G/U$. As a consequence we prove that essential semigroup $\Sigma$ is indeed a semigroup.

 Consider the homogeneous space $G/U$.
 Let $B=T\rightthreetimes U$ be the Borel subgroup. Then

$$
 \mathbb{C}[G/U]=\bigoplus_{\lambda} \mathbb{C}[G]_{\lambda}^{(B)},
$$
where $$\mathbb{C}[G]_{\lambda}^{(B)}=\{f\in\mathbb{C}[G]\mid f(gtu)=\lambda(t)f(g),\, \forall g\in G, t\in T, u\in U\}$$
is the subspace of eigenfunctions of weight $\lambda$ for $B$ acting on $\mathbb{C}[G]$ by right translations of an argument.
 Each subspace $\mathbb{C}[G]_{\lambda}^{(B)}$ is finite-dimensional and is isomorphic as a $G$-module (with respect to the action of $G$ by left translations of an argument), to the space $V(\lambda)^{*}$ of linear functions on $V(\lambda)$ (see \cite{[Ï]}, Theorem 3). The isomorphism is given by the formula:
$$
 V(\lambda)^{*}\ni\omega \longmapsto f_{\omega}\in\mathbb{C}[G]_{\lambda}^{(B)},
 \quad\textrm{where}\quad f_{\omega}(g)=\langle\omega, g\emph{v}_{\lambda}\rangle.
$$

Let $U^{-}$ be the maximal unipotent subgroup such that $\mathop{\mathrm{Lie}}U^{-}=\mathfrak{u^{-}}$. The function $f_{\omega}$ is uniquely determined by its restriction to the dense open subset $U^{-}\cdot$$T\cdot$$U$; moreover
\begin{multline*}
 $$f_{\omega}(u^{-}\cdot t\cdot u)=\langle\omega
,u^{-} tu\emph{v}_{\lambda}\rangle=\langle\omega
,\lambda(t)u^{-}\emph{v}_{\lambda}\rangle=\lambda(t)f_{\omega}(u^{-}),\\ \quad \forall u\in U, u^{-}\in U^{-}, t\in T.
$$
\end{multline*}
Next, $U^{-}=U_{-\alpha_{1}}\cdot\ldots\cdot U_{-\alpha_{N}}$, where
$U_{\alpha}=\{\exp(ze_{\alpha})\mid$ $z\in\mathbb{C}\}$ (see \cite[Sec. X, \S 28.1]{[X]}). Hence

$$
 u^{-}=\exp(z_{1}e_{-\alpha_{1}})\cdot\ldots\cdot\exp(z_{N}e_{-\alpha_{N}}).
$$
Thus we obtain
$$
f_{\omega}(u^{-})=\left\langle\omega,
\exp(z_{1}e_{-\alpha_{1}})\cdot\ldots\cdot\exp(z_{N}e_{-\alpha_{N}})
\cdot \emph{v}_{\lambda}\right\rangle=\sum_{\sigma=(\lambda;p_{1},\dots,p_{N})}
\frac{\prod z_{i}^{p_{i}}}{\prod p_{i}!}
\langle\omega,\emph{v}(\sigma)\rangle.
$$

\begin{proposition}
 A signature $\sigma$ is essential if and only if
$\prod z_{i}^{p_{i}}$ is the least term in
 $f_{\omega}|_{U^{-}}$ for some
$\omega \in V(\lambda)^{*}$ in the sense of the order introduced above.
\end{proposition}
\begin{proof}
Let $\prod z_{i}^{p_{i}}$ be the least term in $f_{\omega}|_{U^{-}}$ for some $\omega\in V(\lambda)^{*}$. Then $\omega$ vanishes on all vectors $\emph{v}(\tau)$ with $\tau<\sigma$ and is nonzero at  $\emph{v}(\sigma)$. Consequently, $\emph{v}(\sigma)$ cannot be expressed via $\emph{v}(\tau)$ with $\tau<\sigma$, and hence  $\sigma$ is essential.

Conversely, let $\sigma$ be essential. Consider a function $\omega$ that vanishes on $\emph{v}(\tau)$ for all essential $\tau$ except for $\sigma$. Obviously, $f_{\omega}|_{U^{-}}$ has the desired least term.
\end{proof}
\begin{proposition}
If $\sigma,\tau\in\Sigma$ then $\sigma+\tau\in\Sigma$.
\end{proposition}
\begin{proof}
Suppose that the least terms in $f|_{U^{-}}$ and $g|_{U^{-}}$ correspond to the essential signatures $\sigma$ and $\tau$. Then the least term in $(f\cdot g)|_{U^{-}}$ corresponds to the signature $\sigma+\tau$. Hence $\sigma+\tau$ is essential.
\end{proof}
\subsection{Dyck paths and essential semigroup for $\mathfrak{sl}_{n+1}$}\label{dyck paths} 
We denote by $\varepsilon_i, i=1,\ldots,n+1,$ the weights of the standard representation of $\mathfrak{sl}_{n+1}$ in $\mathbb{C}^{n+1}$. Here $\varepsilon_i$ is the weight of the vector $e_i$ with respect to Cartan subalgebra of $\mathfrak{sl}_{n+1}$ consisting of diagonal matrices. Let $\beta_i=\varepsilon_i-\varepsilon_{i+1}, i=1,\ldots,n$, be the simple roots and let $\omega_i=\varepsilon_1+\ldots+\varepsilon_i$ be the fundamental weights. Choose the following ordering of positive roots of $\mathfrak{sl}_{n+1}$: $\varepsilon_i-\varepsilon_j>\varepsilon_k-\varepsilon_l$ if either  $i<k$ or $i=k$ and $j>l$. We will use degree reverse lexicographic order on $\mathbb{Z}^{\frac{(n+1)n}{2}}$. Let $\Sigma$ be the corresponding essential semigroup and $\Sigma(\lambda)$ be the set of all signatures in $\Sigma$ with the highest weight $\lambda$.

Following \cite{[FFL1]} we describe $\Sigma$ in terms of \emph{Dyck paths}. A Dyck path is a sequence of negative roots
$$
\textbf{p}=(\beta(0),\beta(1),\ldots,\beta(k)), k\geq0,
$$
where the elements satisfy the following condition:
$$
\textrm{if}\quad \beta(s)=-(\varepsilon_p-\varepsilon_q)\quad\textrm{then}
$$
$$
 \beta(s+1)=-(\varepsilon_{p+1}-\varepsilon_q)\quad\textrm{or}\quad
\beta(s+1)=-(\varepsilon_p-\varepsilon_{q+1}).
$$
Every negative root corresponds to an entry of $(n+1)\times (n+1)$ matrix below the main diagonal. So a Dyck path is a sequence of below-diagonal entries such that each successive element moves strictly rightward or downward relative to the previous element. 

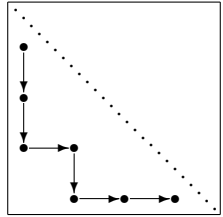
\begin{center}
\begin{picture}(80,80)
\put(0,0){\line(1,0){80}} \put(0,80){\line(1,0){80}}
\put(0,0){\line(0,1){80}} \put(80,0){\line(0,1){80}}
\multiput(3,77)(3,-3){26}{\circle*{1}}
\put(6,63){\circle*{3}}
\put(6,61){\vector(0,-1){15}}

\put(6,44){\circle*{3}}
\put(6,42){\vector(0,-1){15}}

\put(6,25){\circle*{3}}
\put(8,25){\vector(1,0){15}}

\put(25,25){\circle*{3}}
\put(25,23){\vector(0,-1){15}}

\put(25,6){\circle*{3}}
\put(27,6){\vector(1,0){15}}

\put(44,6){\circle*{3}}
\put(46,6){\vector(1,0){15}}

\put(63,6){\circle*{3}}

\end{picture}
\end{center}

A signature $\sigma\in\Sigma$ can be seen as a $(n+1)\times (n+1)$ matrix with zero entries at the main diagonal and above. More precisely the $(j,i)$-th entry ($i<j$) equals to the exponent $p_{j,i}$ of $\sigma$ corresponding to the root $\varepsilon_i-\varepsilon_j$. Denote by $\Sigma^{\mathbb{Q}}$ the rational cone generated by $\Sigma$. The following theorem is proved in \cite{[FFL1]}:
\begin{theor}
Let $\lambda=\sum_{i=1}^{n}k_i\omega_i$ be a dominant weight of $\SL_{n+1}$. The cone $\Sigma^{\mathbb{Q}}$ is given by the following inequalities:
\begin{enumerate}
\item[0.]$p_{i,j}\geq 0,$
\item[1.]$\sum_{(i,j)\in D}p_{i,j}\leq k_s+\ldots+k_{l-1}$

 for every Dyck path $D$ that starts at the $(s,s)$-th entry and ends at the $(l,l)$-th entry.
\end{enumerate}
The essential semigroup $\Sigma$ is the semigroup of integral points in $\Sigma^{\mathbb{Q}}$ and it is generated by $\Sigma(\omega_i)$.
\end{theor}

\section{The branching algebra and the branching semigroup}
\subsection{The branching problem and the branching algebra}\label{branching problem}
Let $\mathfrak{h}\subset\mathfrak{g}$ be a simple Lie subalgebra of $\mathfrak{g}$. Let $H\subset G$ be a connected algebraic group such that $\mathop{\mathrm{Lie}}H=\mathfrak{h}$.

 Restrict the irreducible representation $V_{\mathfrak{g}}(\lambda)$ with the highest weight $\lambda$ to $\mathfrak{h}$:
$$
V_{\mathfrak{g}}(\lambda)|_{\mathfrak{h}}=\bigoplus_{\lambda'}m_{\lambda,\lambda'}V_{\mathfrak{h}}(\lambda'),
$$
where $V_{\mathfrak{h}}(\lambda')$ is the irreducible representation of $\mathfrak{h}$ with the highest weight $\lambda'$, and $m_{\lambda,\lambda'}$ is the multiplicity. The classical branching problem is to determine $m_{\lambda,\lambda'}$.

Consider the action of $H$ on $\mathbb{C}[G/U]=\bigoplus_{\lambda}V(\lambda)^*$ by left translations of an argument. Let $U_{\mathfrak{h}}$ be a maximal unipotent subgroup of $H$. The algebra $\mathbb{C}[G/U]^{U_{\mathfrak{h}}}$ of $U_{\mathfrak{h}}$-invariants is called \emph{the branching algebra}. This is a finitely generated algebra consisting of the highest vectors of $\mathfrak{h}$. A description of this algebra in terms of generators and relations solves the branching problem.

\subsection{The branching semigroup}\label{branching semigroup}

Now we want to introduce a subsemigroup $\Sigma_{\mathfrak{h}}$ of the essential semigroup $\Sigma$, which is related to the branching problem. We call $\Sigma_{\mathfrak{h}}$ \emph{the branching semigroup}.

Recall that $f_\omega\in V(\lambda)^*$ is uniquely determined by its restriction to $U^-\cdot T$. Let $t_1,\ldots,t_n$ be the coordinates on $T$ corresponding to the fundamental weights $\omega_i$, i.e. $t_i=\omega_i(t), t\in T$. Then $f_\omega$ can be thought as a polynomial in $t_1,\ldots,t_n, z_1,\ldots, z_N$. Indeed, if $\lambda=\sum_i k_i\omega_i$ then $$f_\omega(u^-\cdot t)=t_1^{k_1}\cdot\ldots\cdot t_n^{k_n}\cdot \left(\sum_{\sigma=(\lambda;p_{1},\dots,p_{N})}
\frac{\prod z_{i}^{p_{i}}}{\prod p_{i}!}
\langle\omega,\emph{v}(\sigma)\rangle\right).$$
The expression in the brackets has the form $cz_1^{p_1}\cdot\ldots\cdot z_N^{p_N}+\mbox{higher terms}$ (with respect to the chosen monomial order on $\mathbb{Z}^N$), where $c\in\mathbb{C}\backslash \{0\}$. Set $\sign(f_\omega)=(\lambda;p_1,\ldots,p_N)\in\Sigma$. Obviously, $\sign(f_{\omega_1} f_{\omega_2})=\sign(f_{\omega_1})+\sign(f_{\omega_2})$.

Let $\Sigma_{\mathfrak{h}}=\{\sign(f_\omega)\mid f_\omega\in\mathbb{C}[G/U]^{U_{\mathfrak{h}}^{-}}\}$, where $U_{\mathfrak{h}}^{-}\subset H$ is the opposite maximal unipotent subgroup to $U_{\mathfrak{h}}$.  So, $\Sigma_{\mathfrak{h}}$ consists of essential signatures that are the least terms of the lowest vectors with respect to $\mathfrak{h}$. Denote by $\Sigma_{\mathfrak{h}}(\lambda)$ the set of all signatures of the highest weight $\lambda$ in $\Sigma_{\mathfrak{h}}$.

 If $\Sigma_{\mathfrak{h}}$ is finitely generated then a description of $\Sigma_{\mathfrak{h}}$ in terms of generators and relations solves the branching problem. Indeed, the signature $\sigma\in\Sigma_{\mathfrak{h}}(\lambda)$ defines the irreducible representation $V_{\mathfrak{h}}(\lambda')$ in $V_{\mathfrak{g}}(\lambda)$ where $\lambda'$ is the weight of $v(\sigma)$ restricted to $\mathfrak{h}$. Therefore the multiplicity
  $m_{\lambda,\lambda'}$ is equal to the number of signatures $\sigma$ in $\Sigma_{\mathfrak{h}}(\lambda)$ such that the weight of $v(\sigma)$ is $\lambda'$ when restricted to $\mathfrak{h}$.

\subsection{ The branching rule for $\SL_{n+m}\supset\SL_{n}\times\SL_{m}$}
 Choose the ordering of positive roots of $\SL_{n+m}$ and the monomial order as in \ref{dyck paths}. For every signature $\sigma=(\lambda;\bold{p})$ let $p_{i,j}, i>j,$ be the exponent of $\sigma$ corresponding to the root $\varepsilon_j-\varepsilon_i$.
A signature $\sigma$ can be thought as a filling an $(n+m)\times (n+m)$ matrix with numbers $p_{i,j}$ below the main diagonal.

 Denote by $\Sigma_{l,s}$ the branching semigroup for $\SL_{n+m}\supset\SL_{l}\times\SL_{s}$, where $\SL_{l}$ and $\SL_{s}$ are embedded in $\SL_{n+m}$ as the upper left corner and the lower right corner, respectively. We will see that a signature $\sigma\in\Sigma_{n,m}$ may have nonzero exponents only in lower left $m\times n$ corner. We set $v_{i,j}=p_{i,j}+p_{i-1,j}+\ldots+p_{j+1,j}-p_{i-1,j+1}-p_{i-2,j+1}-\ldots-p_{j+2,j+1}$ and $h_{i,j}=p_{i,j}+p_{i,j+1}+\ldots+p_{i,i-1}-p_{i-1,j+1}-\ldots-p_{i-1,i-2}$. Here $v$ and $h$ stand for 'vertical' and 'horizontal', respectively.
\begin{theor}{\label{main theorem}}
Let $\lambda=\sum_{i=1}^{n+m-1}k_i\omega_i$ be a dominant weight of $\SL_{n+m}$. The cone $\Sigma^{\mathbb{Q}}_{n,m}$ is given by the following inequalities and equalities:
\begin{enumerate}
\item[0.] $p_{i,j}\geq 0$,
\item[1.] $p_{i,j}=0,\quad i<n+1\quad \textrm{or}\quad j>n$,
\item[2.] $ v_{i,j}\leq k_j,\quad 1\leq j\leq n-1$,
\item[3.] $ h_{i,j}\leq k_{i-1},\quad n+2\leq i\leq n+m$,
\item[4.] $ p_{n+1,n}\leq k_n.$
\end{enumerate}
The branching semigroup $\Sigma_{n,m}$ is the semigroup of integral points in $\Sigma^{\mathbb{Q}}_{n,m}$.
\end{theor}
\begin{remark}{\label{remark 3}}
\emph{The inequalities for $\Sigma^{\mathbb{Q}}_{n,m}$ imply the inequalities for $\Sigma^{\mathbb{Q}}$. Indeed, if $D$ is a Dyck path that bends only at the $(l,s)$-th entry  then the inequality $\sum_{(i,j)\in D}p_{i,j}\leq k_{s}+\ldots+k_{l-1}$ can be seen as $v_{l,s}+v_{l,s+1}+\ldots+v_{l,l-1}\leq k_s+\ldots+k_{l-1}$. Obviously, the last inequality is the sum of inequalities $v_{l,j}\leq k_j$ for $j=s,\ldots,l-1$. Similar arguments applicable for an arbitrary Dyck path $D$.}
\end{remark}
\begin{example}
\emph{Let $\lambda=k_1\omega_1+k_2\omega_2+k_3\omega_3$ be a dominant weight of $\SL_{4}$. The cone $\Sigma^{\mathbb{Q}}_{2,2}$ is given by the following inequalities and equalities:}
\begin{enumerate}
\item $p_{2,1}=0,\\
 p_{4,3}=0,$
\item $v_{3,1}=p_{3,1}\leq k_1,\\
 v_{4,1}=p_{4,1}+p_{3,1}-p_{3,2}\leq k_1,$
\item $h_{4,1}=p_{4,1}+p_{4,2}-p_{3,2}\leq k_3,\\
 h_{4,2}=p_{4,2}\leq k_3,$
\item $p_{2,2}\leq k_2.$
\end{enumerate}
\end{example}
Theorem \ref{main theorem} is a special case of a more general theorem stated below:
\begin{theor}{\label{general theorem}}
Let $\lambda=\sum_{i=1}^{n+m-1}k_i\omega_i$ be a dominant weight of $\SL_{n+m}$. The cone $\Sigma^{\mathbb{Q}}_{l,s}$ is given by the following inequalities and equalities:
\begin{enumerate}
\item[0.] $p_{i,j}\geq 0$,
\item[1.] $p_{i,j}=0,\quad i<l+1\quad \textrm{or}\quad j>n+m-s$,
\item[2.] $ v_{i,j}\leq k_j,\quad 1\leq j\leq l-1$,
\item[3.] $ h_{i,j}\leq k_{i-1},\quad n+m-s+2\leq i\leq n+m$,
\item[4.] $(\lambda;\bold{p})\in\Sigma$, i.e. must satisfy all inequalities defining $\Sigma^{\mathbb{Q}}.$
\end{enumerate}
The branching semigroup $\Sigma_{l,s}$ is the semigroup of integral points in $\Sigma^{\mathbb{Q}}_{l,s}$.
\end{theor}
Denote by $\tilde{\Sigma}_{l,s}$ all integer solutions of the system of inequalities and equalities from Theorem \ref{general theorem}. Theorem \ref{general theorem} says that $\tilde{\Sigma}_{l,s}=\Sigma_{l,s}$.
\begin{remark}\label{transpose}
\emph{There is an obvious symmetry in the inequalities defining $\Sigma_{l,s}$ and $\Sigma_{s,l}$. Recall that a signature $\sigma$ can be seen as a $(n+m)\times(n+m)$ matrix. It is easy to see that if $\sigma\in\Sigma_{l,s}(\lambda)$ then $\sigma^{t}\in\Sigma_{s,l}(\lambda^*)$, where $t$ sends a matrix to its transpose with respect to the secondary diagonal and $\lambda^*$ denotes the highest weight of $V(\lambda)^*$. This symmetry comes from the automorphism of $\SL_{n+m}$ that sends a matrix $A$ to $(A^{-1})^t$ and swaps $\SL_{l}\times\SL_{s}$ and $\SL_{s}\times\SL_{l}$. Moreover, if $\sigma$ corresponds to a representation (of $\SL_l\times\SL_s$) with the highest weight $\mu$ with respect to the maximal torus $T$ in $\SL_{m+n}$ then $\sigma^t$ corresponds to a representation (of $\SL_s\times\SL_l$) with the highest weight $\mu^*$. In particular, the dimensions of these two representations coincide.}
\end{remark}

\begin{remark}
\emph{Similar to Remark \ref{remark 3} one can show that in $4.$ it is enough to consider those Dyck paths $D$ that starts at the $(r,r)$ entry and ends at the $(t,t)$ entry with $r\geq l$ and $t\leq n+m-s+1.$ In particular, Theorem \ref{general theorem} implies Theorem \ref{main theorem}.}
\end{remark}
\begin{lemma}\label{lemma 1}
Let $\lambda=\sum_{i=1}^{l-1}k_i\omega_i$ be a dominant weight of $\SL_l$. Then $\Sigma_{l-1,1}$ is the semigroup of integral points satisfying the following inequalities and equalities:
\begin{enumerate}
\item $p_{i,j}=0,\quad i\neq l,$
\item $p_{l,j}\leq k_j,\quad j=1,\ldots,l-1$.
\end{enumerate}
Moreover, the semigroup $\Sigma_{l-1,1}$ is generated by $\Sigma_{l-1,1}(\omega_i)$.
\end{lemma}
\begin{proof}
This lemma is proved in \cite{[G3]}.
\end{proof}
In particular, this lemma says that $\Sigma_{l-1,1}=\tilde{\Sigma}_{l-1,1}$. 
\begin{sled}\label{corollary 1}
$\Sigma_{1,l-1}=\tilde{\Sigma}_{1,l-1}$.
\end{sled}
\begin{proof}
We have an isomorphism between $\tilde{\Sigma}_{l-1,1}$ and $\tilde{\Sigma}_{1,l-1}$ (see Remark \ref{transpose}) sending $\sigma\in\tilde{\Sigma}_{l-1,1}(\lambda)$ to $\sigma^t\in\tilde{\Sigma}_{1,l-1}(\lambda^*)$. One can easily check that $\Sigma_{1,l-1}(\omega_i)=\tilde{\Sigma}_{1,l-1}(\omega_i)$. Since $\Sigma_{l-1,1}(\omega_i)$ generates $\Sigma_{l-1,1}$ we see that $\Sigma_{1,l-1}(\omega_i^*)$ ($=\tilde{\Sigma}_{1,l-1}(\omega_i^*)$) generate $\tilde{\Sigma}_{1,l-1}$. Therefore $\tilde{\Sigma}_{1,l-1}\subset\Sigma_{1,l-1}$. The equality follows from dimension arguments: $\dim V(\lambda)=\dim V(\lambda^*)$ and hence $|\Sigma_{l-1,1}(\lambda)|=|\tilde{\Sigma}_{1,l-1}(\lambda^*)|=|\Sigma_{1,l-1}(\lambda^*)|$, where in the last equality we use the fact that $\sigma$ and $\sigma^t$ correspond to representations of the same dimension.
\end{proof}
\begin{remark}
\emph{In what follows it will be convenient for us to understand an element of $\Sigma_{l-1,1}(\lambda)$ with some fixed highest weight $\lambda$ as a tuple $\textbf{t}=(t_1,\ldots,t_{l-1}),$ where $t_i=p_{l,i}$. Then if one restricts a representation $V_{\SL_l}(\lambda)$ to $\SL_{l-1}$ the tuples $\textbf{t}\in\Sigma_{l-1,1}$ parametrize the irreducible representations of $\SL_{l-1}$ in $V_{\SL_l}(\lambda)|_{\SL_{l-1}}$. One has}
$$
V_{\SL_{l}}(\lambda)|_{\SL_{l-1}}=\bigoplus_{\textbf{t}\leq \lambda}V_{\SL_{l-1}}(\lambda-t_1\varepsilon_1-\ldots-t_{n-1}\varepsilon_{n-1}),
$$
\emph{where $\textbf{t}\leq\lambda$ means $t_i\leq k_i, \lambda=\sum_{i=1}^{l-1} k_i\omega_i$}.
\end{remark}

The next lemma explains why all exponents of $\sigma\in\Sigma_{l,s}$ corresponding to all roots of $\SL_{l}\times\SL_{s}$ are zero, i.e. it explains the equalities of type $1$ in Theorem \ref{general theorem}. We prove a more general result.

Let $\mathfrak{h}\subset\mathfrak{g}$ be a regular embedding of a semisimple Lie algebra $\mathfrak{h}$ such that $\mathfrak{t}_\mathfrak{h}\subset \mathfrak{t}_\mathfrak{g}$ and $\Delta_+(\mathfrak{h})\subset\Delta_+(\mathfrak{g})$. Let $\Sigma=\Sigma_{\mathfrak{g}}$ be an essential semigroup corresponding to some \emph{homogeneous} monomial order, and let $\Sigma_{\mathfrak{h}}$ be the corresponding branching semigroup.
\begin{lemma}
In the above setting if $\sigma\in\Sigma_{\mathfrak{h}}$ then all the exponents of $\sigma$ corresponding to the roots in $\Delta_{+}(\mathfrak{h})$ are zero.
\end{lemma}
\begin{proof}
Let $\sigma\in\Sigma_{\mathfrak{h}}$. This means that $\sigma=\sign(v_{\mu}^*)$, where $V^*_{\mathfrak{h}}(\mu)\subset V^*_{\mathfrak{g}}(\lambda)$ and $v_{\mu}^*$ is the lowest vector in $V^*_{\mathfrak{h}}(\mu)$. The signature $\sigma$  is the minimal signature in $\Sigma$ satisfying $\langle v^*_{\mu},v(\sigma)\rangle\neq 0.$ The vector $v(\sigma)$ is a weight vector and has the $\mathfrak{h}$-invariant projection $c\cdot v_{\mu}$ to $V_{\mathfrak{h}}(\mu)$, where $v_{\mu}$ is the highest vector and $c\in\mathbb{C}$ is nonzero. Now suppose that $\sigma$ has a nonzero exponent corresponding to a root $\alpha\in\Delta_{\mathfrak{h}}$. Since the monomial order is homogeneous we can assume that $\alpha$ is the first root in the ordering of positive roots $\{\alpha=\alpha_1,\ldots,\alpha_N\}$ which we use to define $\Sigma$. Then one has $v(\sigma)=e_{-\alpha}v(\sigma')$ for some signature $\sigma'$. Then weight argument show that $v(\sigma')$ has zero the $\mathfrak{h}$-invariant projection to $V_{\mathfrak{h}}(\mu)$. This implies that the $\mathfrak{h}$-invariant projection of $v(\sigma)$ is zero as well since $e_{-\alpha}\in\mathfrak{h}$. A contradiction.
\end{proof}
One has 
$$
V_{\SL_{n+m}}(\lambda)|_{\SL_{l}\times\SL_{s}}=\bigoplus_{\sigma\in\Sigma_{l,s}(\lambda)} V_{\SL_{l}\times\SL_{s}}(\wt(\sigma)).
$$  
If we restrict the representation $V_{\SL_{l}\times\SL_{s}}(\wt(\sigma))$ further to $\SL_{l-1}\times\SL_{s}$ then according to lemma \ref{lemma 1} 
$$
V_{\SL_{l}\times\SL_{s}}(\wt(\sigma))|_{\SL_{l-1}\times\SL_{s}}=\bigoplus_{\bf{t}\leq\wt(\sigma)}V_{\SL_{l-1}\times\SL_{s}}(\wt(\sigma)-t_1\varepsilon_1-\ldots-t_{l-1}\varepsilon_{l-1}),
$$
where $\textbf{t}\leq \wt(\sigma)$ means $t_i\leq l_i, \wt(\sigma)=\sum_{i=1}^{l-1}l_i\omega_i$.

 Note also that according to the definition of $\wt(\sigma)$ one has 
 $$
 l_i=k_i-p_{1,i}-\ldots-p_{k,i}+p_{1,i+1}+\ldots+p_{k,i+1}, i=1,\ldots,l-1.
 $$
 This motivates the following notation:
$$
\Sigma^{t}_{l,s}(\lambda)=\{(\sigma,\textbf{t})\mid \sigma\in\Sigma_{l,s}(\lambda), t_i\leq k_i-p_{1,i}-\ldots-p_{k,i}+p_{1,i+1}+\ldots+p_{k,i+1}, i=1,\ldots,l-1\}.
$$ 
  Obviously, $\Sigma^{t}_{l,s}:=\cup_{\lambda}\Sigma^{t}_{l,s}(\lambda)$ is a subsemigroup in $\Sigma_{l,s}\times\mathbb{Z}_{\geq 0}^{l-1}$. In such a pair $(\sigma,\textbf{t})\in\Sigma^{t}_{l,s}(\lambda)$ the irreducible representation of $\SL_{l}\times\SL_{s}$ is parametrized by $\sigma$ and further restriction to $\SL_{l-1}\times\SL_{s}$ is controlled by $\bf{t}$.

\begin{proof}[Proof of Theorem \ref{general theorem}]

Recall that $\tilde{\Sigma}_{l,s}$ denotes all integral solutions of the system of inequalities and equalities from Theorem \ref{general theorem}. Let $\tilde{\Sigma}_{l,s}^{t}$ be a semigroup constructed in the same way as $\Sigma_{l,s}^{t}$. It is easy to see that $\tilde{\Sigma}_{l,s}=\tilde{\Sigma}_{l,1}\cap\tilde{\Sigma}_{1,s}$. We want to show that $\tilde{\Sigma}_{l,s}=\Sigma_{l,s}$. The main ingredient in our proof of Theorem \ref{general theorem} is the following lemma:
\begin{lemma}\label{main lemma}
There exists an injective map $\phi_{l,s}: \tilde{\Sigma}_{l,s}^{t}\to\tilde{\Sigma}_{l-1,s}$.
\end{lemma}
We postpone the proof to \ref{proof of lemma}. We fix $s$ and argue by induction on $l$. The base case of the induction is given by 
Lemma \ref{lemma 1} and Corollary \ref{corollary 1}. 
 
Obviously, $\Sigma_{l,s}$ is contained in $\Sigma_{l,1}\cap\Sigma_{1,s}$ since every highest weight vector with respect to $\SL_{l}\times\SL_{s}$ is a highest weight vector with respect to both $\SL_{l}\times\{e\}$ and $\{e\}\times\SL_{s}$. By Lemma \ref{lemma 1} and Corollary \ref{corollary 1} one has $\Sigma_{l,s}\subset\tilde{\Sigma}_{l,s}=\Sigma_{l,1}\cap\Sigma_{1,s}$.

The semigroup $\Sigma_{l,s}$ describes the restriction from $\SL_{m+n}$ to $\SL_l\times\SL_s$. The further restriction to $\SL_{l-1}\times\SL_{s}$ is controlled by $\Sigma_{l,s}^t$ or $\Sigma_{l-1,s}$. In particular for every dominant weight $\lambda$ we have $|\Sigma_{l,s}^t(\lambda)|=|\Sigma_{l-1,s}(\lambda)|$. The embedding $\Sigma_{l,s}\subset\tilde{\Sigma}_{l,s}$ implies $\Sigma_{l,s}^t\subset\tilde{\Sigma}_{l,s}^t$ and we conclude that $|\tilde{\Sigma}_{l,s}^t(\lambda)|\geq |\Sigma_{l,s}^t(\lambda)|=|\Sigma_{l-1,s}(\lambda)|=|\tilde{\Sigma}_{l-1,s}(\lambda)|$, where the last equality holds by the induction hypothesis. On the other hand, $|\tilde{\Sigma}_{l,s}^t(\lambda)|\leq |\tilde{\Sigma}_{l-1,s}(\lambda)|$ by Lemma \ref{main lemma} and we conclude that $|\tilde{\Sigma}_{l,s}^t(\lambda)|= |\Sigma_{l-1,s}(\lambda)|$. Therefore, $|\Sigma_{l-1,s}(\lambda)|=|\Sigma_{l,s}^t(\lambda)|\leq |\tilde{\Sigma}_{l,s}^t(\lambda)|= |\Sigma_{l-1,s}(\lambda)|$ and hence $|\Sigma_{l,s}^t(\lambda)|= |\tilde{\Sigma}_{l,s}^t(\lambda)|$. The last equality implies $\tilde{\Sigma}_{l,s}=\Sigma_{l,s}$.
\end{proof}

 \subsection{Proof of Lemma \ref{main lemma}}\label{proof of lemma}
We recall the notations. $\tilde{\Sigma}_{l,s}\subset\Lambda^+\times\mathbb{Z}^N$ is the semigroup of integral solutions of the inequalities and equalities from Theorem \ref{general theorem}. A signature $\sigma\in\tilde{\Sigma}_{l,s}$ has the form $\sigma=(\lambda;\textbf{p})$, where $\lambda$ is the highest weight of $\sigma$ and $\textbf{p}$ is the set of exponents of $\sigma$. Having fixed the highest weight $\lambda$ we can think that $\sigma$ is a matrix having $p_{i,j}$ as the $(i,j)$-th entry. $\tilde{\Sigma}_{l,s}^t(\lambda)\subset\tilde{\Sigma}_{l,s}(\lambda)\times\mathbb{Z}^{l-1}$ is the set $\{(\sigma,\textbf{t})|\sigma\in\tilde{\Sigma}_{l,s}(\lambda), t_i\leq k_i-p_{1,i}-\ldots-p_{n+m,i}+p_{1,i+1}+\ldots+p_{n+m,i+1}, i=1,\ldots,l-1\}.$ 

To avoid ambiguity, in this section $v_{i,j}$ always denotes
 $$p_{i,j}+p_{i-1,j}+\ldots+p_{l+1,j}-p_{i-1,j+1}-p_{i-2,j+1}-\ldots-p_{l+1,j+1}. $$
 
 In particular, the inequalities 2. for $\tilde{\Sigma}_{l-1,s}$ take the form $v_{i,j}+p_{l,j}-p_{l,j+1}\leq k_j$.

\begin{lemma*}
There exists an injective map $\phi_{l,s}: \tilde{\Sigma}_{l,s}^{t}\to\tilde{\Sigma}_{l-1,s}$.
\end{lemma*}
In the subsection \ref{construction} for $(\sigma,\textbf{t})\in\tilde{\Sigma}_{l,s}^t(\lambda)$ we describe an algorithm constructing a matrix $\phi_{l,s}(\sigma,\textbf{t})$. Then in the subsection \ref{correctness} we prove that the algorithm is correct in the sense that there are no negative entries in the constructed matrix. In the subsection \ref{image} we prove that $\phi_{l,s}(\sigma,\textbf{t})$ belongs to $\tilde{\Sigma}_{l-1,s}$. Finally, the subsection \ref{injectivity} is devoted to the proof of injectivity of $\phi_{l,s}$.
 
\subsubsection{Construction of $\phi_{l,s}$.}\label{construction} Starting from $(\sigma,\textbf{t})$ we describe the process of successive transformations of the matrix $\sigma=\sigma^{(0)}\leadsto\sigma'\leadsto\sigma''\leadsto\ldots\leadsto\sigma^{(M)}.$ The result $\sigma^{(M)}$ is the definition of $\phi_{l,s}(\sigma,\textbf{t})$. We start at row $l+1$ and move from the first element to the $(l-1)$-th, then proceed to the $l+2$-th row and again move from the first element to the $(l-1)$-th, continuing this process and ending at the $(l-1)$-th element of the $(n+m)$-th row. At each entry we somehow transform our matrix (see the algorithm below). If we are now at the $(i,j)$-th entry then we say that we \emph{terminate} the algorithm if we are done with the $(i,j)$-th entry and move to the next entry.

\emph{The left shift} of the $(i,j)$-th entry in $\sigma^{(k)}$ is the transformation $\sigma^{(k)}\leadsto\sigma^{(k+1)}$, where $p^{(k+1)}_{i,j-1}=p^{(k)}_{i,j-1}+1$ and $p^{(k+1)}_{i,j}=p^{(k)}_{i,j}-1.$ We denote by $\sigma_j^{(k)}$ the matrix obtained from $\sigma$ by the same transformations as $\sigma^{(k)}$ but ignoring the left shifts of the entries in $j+2$-th column. Recall that for $\sigma\in\tilde{\Sigma}_{l,s}$ we have $p_{l,\ldots}=0$. First of all we set $p^{(1)}_{l,i}=\min(t_i,\min_{j}(k_i-v^{(0)}_{j,i}))$. So we have $p^{(1)}_{l,i}=t_i$ or there exists $j$ such that $p^{(1)}_{l,i}+v^{(1)}_{j,i}=k_i$. Denote $\tilde{t}^{(1)}_i=t_i-p^{(1)}_{l,i}\geq 0$.

We say that the $(i,j)$-th entry is \emph{the leader} of $j$-th column of $\sigma^{(k)}$ if $v^{(k)}_{i,j}+p_{l,j}^{(k)}\geq k_j$ in $\sigma^{(k)}_j$ and $i$ is the maximal number satisfying $v_{i,j}=\max_{i'}v_{i',j}$.

Now we describe how we transform the current $\sigma^{(k)}$ if we are at the $(i,j)$-th entry.

If the $(i,j)$-th entry is not the leader of $j$-th column of $\sigma^{(k)}$ or $\tilde{t}^{(k)}_j=0$ then we terminate the algorithm. Otherwise, if the $(i,j)$-th entry is the leader of $j$-th column of $\sigma^{(k)}$ and $\tilde{t}^{(k)}_j>0$ then we transform $\sigma^{(k)}\leadsto\sigma^{(k+1)}\leadsto\ldots$ in the following way:
\begin{enumerate}
 \item[1.] $\tilde{t}^{(k+1)}_{j}=\tilde{t}^{(k)}_{j}-1$ and we do the left shift of the $(i,j+1)$-th entry,
 \item[2.] if $j+1=l$ then we terminate the algorithm. Otherwise we find an entry, say $(s_1,j+1)$, in the $j+1$-th column such that 
 \item[(2.a)] $v^{(k+1)}_{s_1,j+1}+p_{l,j+1}\geq k_{j+1}$ in $\sigma^{(k+1)}_{j+1}$ and $s_1<i$ 
 \item[(2.b)] $s_1$ is the maximal number satisfying $$v_{s_1,j+1}=\max_{i'<i}\{v_{i',j+1}\}$$
\\ and we do the left shift of the $(s_1,j+2)$-th entry.
If there is no such $s_1$ then $p^{(k+2)}_{l,j+1}=p^{(k+1)}_{l,j+1}+1$ and terminate the algorithm.
 \item[3.] if $j+2=l$ then we terminate the algorithm. Otherwise we find an entry, say $(s_2,j+2)$, in the $j+2$-th column such that 
 \item[(3.a)] $v^{(k+2)}_{s_2,j+2}+p_{l,j+2}\geq k_{j+2}$ in $\sigma^{(k+2)}_{j+2}$ and $s_2<s_1$
 \item[(3.b)] $s_2$ is the maximal number satisfying $$v_{s_2,j+2}=\max_{i'<s_1}\{v_{i',j+2}\}$$
 \\and we do the left shift of the $(s_2,j+3)$-th entry.
If there is no such $s_2$ then $p^{(k+3)}_{l,j+2}=p^{(k+2)}_{l,j+2}+1$ and terminate the algorithm.
\item[4.] And so on\ldots.
\end{enumerate} 
\begin{remark}
\emph{It is not hard to see that when the algorithm completely terminates we have $\tilde{t}_i=0, i=1,\ldots,l-1.$}
\end{remark}

\subsubsection{Correctness of $\phi_{l,s}$.}\label{correctness}
In this subsection we prove that the algorithm in the previous subsection is correct in the sense that if we do the left shift of the $(i,j)$-th entry then $p_{i,j}>0$.

Now we deal with the left shifts in 1.

Notice that if there is no the leader in the $j$-th column in $\sigma^{(1)}$ then $\tilde{t}^{(1)}_j=0$ and we do not do the left shifts of elements in $j+1$-th column. Otherwise, let the $(i,j)$-th entry be the leader of $j$-th column in $\sigma^{(1)}$. Then we have the inequality
 $$\tilde{t}^{(1)}_j\leq (p^{(1)}_{i,j+1}-p^{(1)}_{i+1,j})+(p^{(1)}_{i+1,j+1}-p^{(1)}_{i+2,j})+\ldots+p^{(1)}_{n+m,j+1}.$$
  Following the algorithm it is not hard to see that for the new leader of $\sigma^{(k)}$ for some $k$, say $(i',j)$, this inequality in $\sigma^{(k)}_{j}$ remains true, i.e. $\tilde{t}^{(k)}_j\leq (p^{(k)}_{i',j+1}-p^{(k)}_{i'+1,j})+(p^{(k)}_{i'+1,j+1}-p^{(k)}_{i'+2,j})+\ldots+p^{(k)}_{n+m,j+1}.$ In particular, if the $(n+m,j)$-th entry is the leader of $j$-th column in $\sigma^{(k)}$ for some $k$ then $\tilde{t}^{(k)}_j\leq p^{(k)}_{n+m,j+1}$ in $\sigma_{j}^{(k)}$ and hence we can do the left shift of $(n+m,j+1)$-th entry if $\tilde{t}_j^{(k)}>0$. If the $(i,j)$-th entry of $\sigma^{(k)}$ for some $k$ is the leader and $i<n+m$ then, obviously, $p^{(k)}_{i,j+1}>p^{(k)}_{i+1,j}\geq 0$ and we can do the left shift. 

Now we deal with the left shifts in 2, 3, and so on. 

Suppose that $p^{(k)}_{s_1,j+2}=0$ in $\sigma^{(k)}_{j+1}$ and we want to do the left shift of the $(s_1,j+2)$-th entry. Then if $s_1\neq i-1$ we have in $\sigma^{(k)}_{j+1}$
$$v^{(k)}_{s_1+1,j+1}+p^{(k)}_{l,j+1}=v^{(k)}_{s_1,j+1}-p^{(k)}_{s_1,j+2}+p^{(k)}_{s_1+1,j+1}+p^{(k)}_{l,j+1}\geq v^{(k)}_{s_1,j+1}+p^{(k)}_{l,j+1}.$$ 
Therefore the $(s_1,j+1)$-th entry does not satisfy 2.b. Otherwise, if $s_1=i-1$  then the $(i-1,j+1)$-th entry can not satisfy 2.a and 2.b (or 3.a, 3.b or 4.a, 4.b and so on). Indeed, before the left shift of the $(i,j+1)$-th entry one has $v^{(k)}_{i,j+1}+p^{(k)}_{l,j+1}>k_{j+1}$ in $\sigma^{(k)}_{j+1}$ and by analyzing possible left shifts of elements in $j+2$-th column we see that there should have occurred the left shifts of some $(\ldots,j+2)$-th entry above $i$-th row and there exists $i'$ such that $v^{(k)}_{i',j+1}\geq v^{(k)}_{i,j+1}$ in $\sigma^{(k)}_{j+1}$ and $i'<i$ (even $i'<i-1$ since $p^{(k)}_{i-1,j+2}=0$). So, since $p^{(k)}_{i-1,j+2}=0$ and $p^{(k)}_{i,j+1}\neq 0$ we have $v^{(k)}_{i-1,j+1}<v^{(k)}_{i,j+1}$ in $\sigma^{(k)}_{j+1}$. Combining this with $v^{(k)}_{i',j+1}\geq v^{(k)}_{i,j+1}$ we obtain $v^{(k)}_{i',j+1}> v^{(k)}_{i-1,j+1}$ and conclude that the $(i-1,j+1)$-th entry does not satisfy 2.a, 2.b (or 3.a, 3.b $\ldots$).

\subsubsection{Image of $\phi_{l,s}$}\label{image}
In this subsection we prove that $\phi_{l,s}(\sigma,\textbf{t})$, where $(\sigma,\textbf{t})\in\tilde{\Sigma}_{l,s}^t$  belongs to $\tilde{\Sigma}_{l-1,s}.$

We show that the inequalities 2. from Theorem \ref{general theorem} are satisfied. These inequalities have the form  $v_{i,j}+p_{l,j}-p_{l,j+1}\leq k_j, j=1,\ldots,l-1$.
Under the left shifts the left-hand side either remains unchanged or is reduced by 1 or is increased by 1. The latter happens when the left shift of the $(k,j+1)$-th entry takes place for some $k<i$, provided that no entries $(m,j)$ with $m\leq i$ have been shifted. Then according to the algorithm (see the definition of the leader and 2.a, 2.b) we have $v^{(k)}_{k,j}>v^{(k)}_{i,j}$ in $\sigma^{(k)}_j$ (and hence in $\sigma^{(k)}$). Therefore if the inequality $v^{(k)}_{k,j}+p^{(k)}_{l,j}-p^{(k)}_{l,j+1}\leq k_j$ is satisfied before the left shift of the $(k,j+1)$-th entry then $v^{(k)}_{i,j}+p^{(k)}_{l,j}-p^{(k)}_{l,j+1}<v^{(k)}_{k,j}+p^{(k)}_{l,j}-p^{(k)}_{l,j+1}\leq k_j$. This implies that raising the left-hand side by 1 preserves the inequality.

 We show that the inequalities 3. from Theorem \ref{general theorem} are satisfied. These inequalities have the form  $h_{i,j}\leq k_{i-1}, n+m-s+2\leq i\leq n+m$.
Under the left shifts the left-hand side either remains unchanged or is reduced by 1 or is increased by 1. The latter happens when the left shift of the $(i-1,j+1)$-th entry takes place, provided that the $(i,j)$-th entry has not been shifted. This implies $p^{(k)}_{i-1,j+1}>p^{(k)}_{i,j}$ in $\sigma^{(k)}_j$ (and hence in $\sigma^{(k)}$ as well). Indeed, otherwise $v^{(k)}_{i,j}\geq v^{(k)}_{i-1,j}$ and hence we do not do the left shift of the $(i-1,j+1)$-th entry. We have $h^{(k)}_{i,j}=p^{(k)}_{i,j}-p^{(k)}_{i-1,j+1}+h^{(k)}_{i,j+1}<h^{(k)}_{i,j+1}\leq k_{i-1}$. This implies that raising the left-hand side by 1 preserves the inequality.

We show that the inequalities 4. from Theorem \ref{general theorem} are satisfied. For a Dyck path $D$ denote $s_{D}=\sum_{(i,j)\in D}p_{i,j}$. 

Let $D$ be a Dyck path that starts at the $l$-th row. The left-hand side of the inequality $s_{D}\leq k_1+\ldots$ is increased by 1 only if we do the left shift of some $(i,n)$-th entry (where $(i,n)\notin D$ and $(i,n-1)\in D$) and there are no the left shifts of the $(j,n-1)$-th entries for $j\leq p$, where $p$ is the maximal number satisfying $(p,n-1)\in D$. We have $v^{(k)}_{i,n-1}>v^{(k)}_{p,n-1}$ in $\sigma^{(k)}_j$ (and hence in $\sigma^{(k)}$). Let $\tilde{D}$ is the Dyck path that starts at $(l+1,n)$-th entry and then goes downward to the $(p,n)$-th entry coinciding with $D$ after. If $(i_1,j_1)$, $(i_2,j_2)$, $\ldots$, $(p,n-1)$, $\dots$ are the entries where $D$ bends then (assuming $v_{l,j}=p_{l,j}$, and if $i_1=l$ then $v_{i_1,j_1}=0$) we have $s^{(k)}_D=p^{(k)}_{l,j_1}+v^{(k)}_{i_1,j_1}+v^{(k)}_{i_1,j_1+1}+\ldots+v^{(k)}_{i_1,j_2-1}+v^{(k)}_{i_2,j_2}+v^{(k)}_{i_2,j_2+1}+\dots+v^{(k)}_{p,n-1}+s^{(k)}_{\tilde{D}}$. Then we consider another Dyck path $D'$ that coincide with $D$ but bends at the $(i,n-1)$-th entry (not at the $(p,n-1)$-th entry as $D$) and then goes downward to $(p,n)$-th entry coinciding with the Dyck path $D$ after. One has
 $$s^{(k)}_{D'}=p^{(k)}_{l,j_1}+v^{(k)}_{i_1,j_1}+v^{(k)}_{i_1,j_1+1}+\ldots+v^{(k)}_{i_1,j_2-1}+v^{(k)}_{i_2,j_2}+v^{(k)}_{i_2,j_2+1}+\dots+v^{(k)}_{i,n-1}+s^{(k)}_{\tilde{D}}>$$
$$ >s^{(k)}_{D}=p^{(k)}_{l,j_1}+v^{(k)}_{i_1,j_1}+v^{(k)}_{i_1,j_1+1}+\ldots+v^{(k)}_{i_1,j_2-1}+v^{(k)}_{i_2,j_2}+v^{(k)}_{i_2,j_2+1}+\dots+v^{(k)}_{p,n-1}+s^{(k)}_{\tilde{D}}.
 $$ 
 This implies that raising the left-hand side of $s^{(k)}_{D}<s^{(k)}_{D'}\leq k_1+\ldots$ by 1 preserves the inequality.

\subsubsection{Injectivity of $\phi_{l,s}$.}\label{injectivity}
We show that $\sigma'=\phi_{l,s}(\sigma,\textbf{t})$ uniquely determines $(\sigma,\textbf{t})$. Let $m_{i,j}$ be the number of the left shifts of the $(i,j)$-th entry. The columns of the matrix $\sigma$ are recovered right-to-left by scanning columns of $\sigma'$ right-to-left starting from $l-1$-th, and rows top-to-bottom within each column. This implies that we know $\sigma'_j$ when dealing with the $j$-th column. For the $(i,j)$-th entry if $p'_{l,j}+v'_{i,j}\leq k_j$ in $\sigma'_j$ then $m_{i,j+1}=0$ according to the algorithm. If there were no the left shifts of the $(i',j+1)$-th entries with $i'<i$ then according to the algorithm every left shift of the $(i,j+1)$-th entry increases $v_{i,j}+p_{l,j}$ by 1 in $\sigma_j$ starting from $k_j$. Therefore $m_{i,j+1}=v'_{i,j}+p'_{l,j}-k_j$ in $\sigma'_j$. Otherwise, if $(i',j+1)$ is the lowest entry in $j+1$-th column that was shifted satisfying $i'<i$, then each the left shift of the $(i,j+1)$-th entry increases the difference $(v_{i,j}+p_{l,j})-(v_{i',j}+p_{l,j})$ in $\sigma_j$ by 1, starting from 0. Therefore $m_{i,j+1}=v'_{i,j}-v'_{i',j}$ in $\sigma'_j$. This allows to restore $\sigma$ by $\sigma'.$ Indeed, one has $p_{i,j}=p'_{i,j}-m_{i,j+1}+m_{i,j}$.

It remains to restore $\textbf{t}=(t_1,\ldots,t_{l-1})$. Denote $m_i=m_{l+1,j}+m_{l+2,j}+\ldots+m_{m+n,j}$. Then according to the algorithm $p_{l,i}+m_{i+1}-m_{i}$ is increased by 1 exactly when $\tilde{t}_{i}$ is decreased by 1. After the algorithm terminates one has $\tilde{t}_{i}=0$ hence $t_i=p'_{l,i}+m_{i+1}-m_i$.

\end{document}